\documentclass[reqno]{amsart}
\usepackage{url}
\usepackage[dvips]{graphicx}

\newtheorem{theorem}{Theorem}
\newtheorem{lem}{Lemma}
\newtheorem{defn}{Definition}
\newtheorem{assumption}{Assumption}
\begin{document}
\bibliographystyle{plain}
\title[Brouwer's fixed point theorem with sequentially at most one fixed point]{Brouwer's fixed point theorem with sequentially at most one fixed point}

\author{Yasuhito Tanaka}
\address{Faculty of Economics, Doshisha University, Kamigyo-ku, Kyoto, 602-8580, Japan}
\email{yasuhito@mail.doshisha.ac.jp}

\date{}

\begin{abstract}
We present a constructive proof of Brouwer's fixed point theorem with sequentially at most one fixed point, and apply it to the mini-max theorem of zero-sum games.
\end{abstract}



\maketitle

\section{Introduction}

It is well known that Brouwer's fixed point theorem can not be constructively proved\footnote{\cite{kel} provided a \emph{constructive} proof of Brouwer's fixed point theorem. But it is not constructive from the view point of constructive mathematics \'{a} la Bishop. It is sufficient to say that one dimensional case of Brouwer's fixed point theorem, that is, the intermediate value theorem is non-constructive. See \cite{br} or \cite{da}.}. Sperner's lemma which is used to prove Brouwer's theorem, however, can be constructively proved. Some authors have presented an approximate version of Brouwer's theorem using Sperner's lemma. See \cite{da} and \cite{veld}.  Thus, Brouwer's fixed point theorem is constructively, in the sense of constructive mathematics \'{a} la Bishop, proved in its approximate version.

Also in \cite{da} Dalen states a conjecture that a uniformly continuous function $f$ from a simplex to itself, with property that each open set contains a point $x$ such that $x\neq f(x)$, which means $|x-f(x)|>0$, and also at every point $x$ on the boundaries of the simplex $x\neq f(x)$, has an exact fixed point. In this note we present a partial answer to Dalen's conjecture.

 Recently \cite{berger} showed that the following theorem is equivalent to Brouwer's fan theorem.
\begin{quote}
Each uniformly continuous function $\varphi$ from a compact metric space $X$ into itself with at most one fixed point and approximate fixed points has a fixed point.
\end{quote}
By reference to the notion of \emph{sequentially at most one maximum} in \cite{berg} we require a stronger condition that a function $\varphi$ has \textit{sequentially at most one fixed point}, and will show the following result.
\begin{quote}
Each uniformly continuous function $\varphi$ from a compact metric space $X$ into itself with \emph{sequentially at most one fixed point} and approximate fixed points has a fixed point,
\end{quote}
without the fan theorem. In \cite{orevkov} Orevkov constructed a computably coded continuous function $f$ from the unit square to itself, which is defined at each computable point of the square, such that $f$ has no computable fixed point. His map consists of a retract of the computable elements of the square to its boundary followed by a rotation of the boundary of the square. As pointed out by Hirst in \cite{hirst}, since there is no retract of the square to its boundary, his map does not have a total extension.

In the next section we present our theorem and its proof. In Section 3, as an application of the theorem we consider the mini-max theorem of two-person zero-sum games.

\section{Theorem and proof}

Let $\mathbf{p}$ be a point in a compact metric space $X$, and consider a uniformly continuous function $\varphi$ from $X$ into itself. According to \cite{da} and \cite{veld} $\varphi$ has an approximate fixed point. It means 
\[\mathrm{For\ each}\ \varepsilon>0\ \mathrm{there\ exists}\ \mathbf{p}\in X\ \mathrm{such\ that}\ |\mathbf{p}-\varphi(\mathbf{p})|<\varepsilon.\]
Since $\varepsilon>0$ is arbitrary,
\[\inf_{\mathbf{p}\in X}|\mathbf{p}-\varphi(\mathbf{p})|=0.\]

The notion that $\varphi$ has at most one fixed point is defined as follows;
\begin{defn}[At most one fixed point]
For all $\mathbf{p}, \mathbf{q}\in X$, if $\mathbf{p}\neq \mathbf{q}$, then $\varphi(\mathbf{p})\neq \mathbf{p}$ or $\varphi(\mathbf{q})\neq \mathbf{q}$.
\end{defn}
Next by reference to the notion of \emph{sequentially at most one maximum} in \cite{berg}, we define the notion that $\varphi$ has \emph{sequentially at most one fixed point} as follows;
\begin{defn}[Sequentially at most one fixed point]
All sequences $(\mathbf{p}_n)_{n\geq 1}$, $(\mathbf{q}_n)_{n\geq 1}$ in $X$ such that $|\varphi(\mathbf{p}_n)-\mathbf{p}_n|\longrightarrow 0$ and $|\varphi(\mathbf{q}_n)-\mathbf{q}_n|\longrightarrow 0$ are eventually close in the sense that $|\mathbf{p}_n-\mathbf{q}_n|\longrightarrow 0$.
\end{defn}

Now we show the following lemma.
\begin{lem}
Let $\varphi$ be a uniformly continuous function from a compact metric space $X$ into itself. Assume $\inf_{\mathbf{p}\in X}|\mathbf{p}-\varphi(\mathbf{p})|=0$. If the following property holds,
\begin{quote}
For each $\delta>0$ there exists $\varepsilon>0$ such that if $\mathbf{p}, \mathbf{q}\in X$, $|\varphi(\mathbf{p})-\mathbf{p}|<\varepsilon$ and $|\varphi(\mathbf{q})-\mathbf{q}|<\varepsilon$, then $|\mathbf{p}-\mathbf{q}|\leq \delta$,
\end{quote}
then, there exists a point $\mathbf{r}\in X$ such that $\varphi(\mathbf{r})=\mathbf{r}$, that is , we have a fixed point of $\varphi$. \label{fix0}
\end{lem}
\begin{proof}
Choose a sequence $(\mathbf{p}_n)_{n\geq 1}$ in $X$ such that $|\varphi(\mathbf{p}_n)-\mathbf{p}_n|\longrightarrow 0$. Compute $N$ such that $|\varphi(\mathbf{p}_n)-\mathbf{p}_n|<\varepsilon$ for all $n\geq N$. Then, for $m, n\geq N$ we have $|\mathbf{p}_m-\mathbf{p}_n|\leq \delta$. Since $\delta>0$ is arbitrary, $(\mathbf{p}_n)_{n\geq 1}$ is a Cauchy sequence in $X$, and converges to a limit $\mathbf{r}\in X$. The continuity of $\varphi$ yields $|\varphi(\mathbf{r})-\mathbf{r}|=0$, that is, $\varphi(\mathbf{r})=\mathbf{r}$.
\end{proof}

Next we show the following theorem.
\begin{theorem}
Each uniformly continuous function $\varphi$ from a compact metric space $X$ into itself with sequentially at most one fixed point and approximate fixed points has a fixed point
\end{theorem}
\begin{proof}
Choose a sequence $(\mathbf{r}_n)_{n\geq 1}$ in $X$ such that $|\varphi(\mathbf{r}_n)-\mathbf{r}_n|\longrightarrow 0$. In view of Lemma \ref{fix0} it is enough to prove that the following condition holds.
\begin{quote}
For each $\delta>0$ there exists $\varepsilon>0$ such that if $\mathbf{p}, \mathbf{q}\in X$, $|\varphi(\mathbf{p})-\mathbf{p}|<\varepsilon$ and $|\varphi(\mathbf{q})-\mathbf{q}|<\varepsilon$, then $|\mathbf{p}-\mathbf{q}|\leq \delta$.
\end{quote}
Assume that the set
\[K=\{(\mathbf{p},\mathbf{q})\in X\times X:\ |\mathbf{p}-\mathbf{q}|\geq \delta\}\]
is nonempty and compact\footnote{See Theorem 2.2.13 of \cite{bv}.}. Since the mapping $(\mathbf{p},\mathbf{q})\longrightarrow \max(|\varphi(\mathbf{p})-\mathbf{p}|,|\varphi(\mathbf{q})-\mathbf{q}|)$ is uniformly continuous, we can construct an increasing binary sequence $(\lambda_n)_{n\geq 1}$ such that
\[\lambda_n=0\Rightarrow \inf_{(\mathbf{p},\mathbf{q})\in K}\max(|\varphi(\mathbf{p})-\mathbf{p}|,|\varphi(\mathbf{q})-\mathbf{q}|)<2^{-n},\]
\[\lambda_n=1\Rightarrow \inf_{(\mathbf{p},\mathbf{q})\in K}\max(|\varphi(\mathbf{p})-\mathbf{p}|,|\varphi(\mathbf{q})-\mathbf{q}|)>2^{-n-1}.\]
It suffices to find $n$ such that $\lambda_n=1$. In that case, if $|\varphi(\mathbf{p})-\mathbf{p}|<2^{-n-1}$, $|\varphi(\mathbf{q})-\mathbf{q}|<2^{-n-1}$, we have $(\mathbf{p},\mathbf{q})\notin K$ and $|\mathbf{p}-\mathbf{q}|\leq \delta$. Assume $\lambda_1=0$. If $\lambda_n=0$, choose $(\mathbf{p}_n, \mathbf{q}_n)\in K$ such that $\max(|\varphi(\mathbf{p}_n)-\mathbf{p}_n|, |\varphi(\mathbf{q}_n)-\mathbf{q}_n|)<2^{-n}$, and if $\lambda_n=1$, set $\mathbf{p}_n=\mathbf{q}_n=\mathbf{r}_n$. Then, $|\varphi(\mathbf{p}_n)-\mathbf{p}_n|\longrightarrow 0$ and $|\varphi(\mathbf{q}_n)-\mathbf{q}_n|\longrightarrow 0$, so $|\mathbf{p}_n-\mathbf{q}_n|\longrightarrow 0$. Computing $N$ such that $|\mathbf{p}_N-\mathbf{q}_N|<\delta$, we must have $\lambda_N=1$. We have completed the proof.
\end{proof}

\section{Application: Minimax theorem of zero-sum games}

Consider a two person zero-sum game. There are two players $A$ and $B$. Player $A$ has $m$ alternative pure strategies, and the set of his pure strategies is denoted by $S_A=\{a_1, a_2, \dots, a_m\}$. Player $B$ has $n$ alternative pure strategies, and the set of his pure strategies is denoted by $S_B=\{b_1, b_2, \dots, b_n\}$. $m$ and $n$ are finite natural numbers. The payoff of player $A$ when a combination of players' strategies is $(a_i, b_j)$ is denoted by $M(a_i, b_j)$. Since we consider a zero-sum game, the payoff of player $B$ is equal to $-M(a_i, b_j)$. Let $p_i$ be a probability that $A$ chooses his strategy $a_i$, and $q_j$ be a probability that $B$ chooses his strategy $b_j$. A mixed strategy of $A$ is represented by a probability distribution over $S_A$, and is denoted by $\mathbf{p}=(p_1, p_2, \dots, p_m)$ with $\sum_{i=1}^mp_i=1$. Similarly, a mixed strategy of $B$ is denoted by $\mathbf{q}=(q_1, q_2, \dots, q_n)$ with $\sum_{j=1}^nq_j=1$. A combination of mixed strategies $(\mathbf{p}, \mathbf{q})$ is called a \emph{profile}. The expected payoff of player $A$ at a profile $(\mathbf{p}, \mathbf{q})$ is written as follows,
\[M(\mathbf{p}, \mathbf{q})=\sum_{i=1}^m\sum_{j=1}^np_iM(a_i,b_j)q_j.\]
We assume that $M(a_i,b_j)$ is finite. Then, since $M(\mathbf{p}, \mathbf{q})$ is linear with respect to probability distributions over the sets of pure strategies of players, it is a uniformly continuous function. The expected payoff of $A$ when he chooses a pure strategy $a_i$ and $B$ chooses a mixed strategy $\mathbf{q}$ is $M(a_i, \mathbf{q})=\sum_{j=1}^nM(a_i, b_j)q_j$, and his expected payoff when he chooses a mixed strategy $\mathbf{p}$ and $B$ chooses a pure strategy $b_j$ is $M(\mathbf{p}, b_j)=\sum_{i=1}^mp_iM(a_i, b_j)$. The set of all mixed strategies of $A$ is denoted by $P$, and that of $B$ is denoted by $Q$. $P$ is an $m-1$-dimensional simplex, and $Q$ is an $n-1$-dimensional simplex.

 We call $v_A(\mathbf{p})=\inf_{\mathbf{q}}M(\mathbf{p}, \mathbf{q})$ the \emph{guaranteed payoff} of $A$ at $\mathbf{p}$. And we define $v_A^*$ as follows,
\begin{eqnarray*}
v_A^*=\sup_\mathbf{p} \inf_\mathbf{q}M(\mathbf{p},\mathbf{q})\label{e1}
\end{eqnarray*}
This is a constructive version of the maximin payoff. Similarly, we call $v_B(\mathbf{q})=\sup_{\mathbf{p}}M(\mathbf{p}, \mathbf{q})$ the guaranteed payoff of player $B$ at $\mathbf{q}$, and define $v_B^*$ as follows,
\begin{eqnarray*}
v_B^*=\inf_\mathbf{q} \sup_\mathbf{p}M(\mathbf{p},\mathbf{q}).
\end{eqnarray*}
This is a constructive version of the minimax payoff. For a fixed $\mathbf{p}$ we have $\inf_{\mathbf{q}}M(\mathbf{p}, \mathbf{q})\leq M(\mathbf{p}, \mathbf{q})\ \mathrm{for\ all}\ \mathbf{q}$, and so
\[\sup_{\mathbf{p}} \inf_{\mathbf{q}}M(\mathbf{p}, \mathbf{q})\leq \sup_{\mathbf{p}} M(\mathbf{p}, \mathbf{q})\ \mathrm{for\ all}\ \mathbf{q}\]
holds. Then, we obtain $\sup_{\mathbf{p}} \inf_{\mathbf{q}}M(\mathbf{p}, \mathbf{q})\leq \inf_{\mathbf{q}}\sup_{\mathbf{p}} M(\mathbf{p}, \mathbf{q})$. This is rewritten as
\begin{equation}
v_A^*\leq v_B^*.\label{e13}
\end{equation}

Define a function $\Gamma=(\mathbf{p}'(\mathbf{p}, \mathbf{q}), \mathbf{q}'(\mathbf{p}, \mathbf{q}))$ as follows;
\[p_i'(\mathbf{p}, \mathbf{q})=\frac{p_i+\max(M(a_i, \mathbf{q})-M(\mathbf{p}, \mathbf{q}),0)}{1+\sum_{k=1}^m\max(M(a_k, \mathbf{q})-M(\mathbf{p}, \mathbf{q}),0)},\]
\[q_j'(\mathbf{p}, \mathbf{q})=\frac{q_j+\max(M(\mathbf{p}, \mathbf{q})-M(\mathbf{p},b_j),0)}{1+\sum_{k=1}^n\max(M(\mathbf{p}, \mathbf{q})-M(\mathbf{p},b_k),0)}.\]

We assume the following condition;
\begin{assumption}
All sequences $((\mathbf{p}_n, \mathbf{q}_n))_{n\geq 1}$, $((\mathbf{p}'_n, \mathbf{q}'_n))_{n\geq 1}$ in $P\times Q$ such that $\max(M(a_i, \mathbf{q}_n)-M(\mathbf{p}_n, \mathbf{q}_n),0)\longrightarrow 0$, $\max(M(\mathbf{p}_n, \mathbf{q}_n)-M(\mathbf{p}_n,b_j),0)\longrightarrow 0$, $\max(M(a_i, \mathbf{q}'_n)-M(\mathbf{p}'_n, \mathbf{q}'_n),0)\longrightarrow 0$ and $\max(M(\mathbf{p}'_n, \mathbf{q}'_n)-M(\mathbf{p'}_n,b_j),0)\longrightarrow 0$ for all $i$ and $j$ are eventually close in the sense that $|(\mathbf{p}_n, \mathbf{q}_n)-(\mathbf{p}'_n, \mathbf{q}'_n)|\longrightarrow 0$.\label{assump1}
\end{assumption}
Under this assumption, we find
\begin{quote}
All sequences $((\mathbf{p}_n, \mathbf{q}_n))_{n\geq 1}$, $((\mathbf{p}'_n, \mathbf{q}'_n))_{n\geq 1}$ in $P\times Q$ such that $|\Gamma((\mathbf{p}_n, \mathbf{q}_n))-(\mathbf{p}_n, \mathbf{q}_n)|\longrightarrow 0$ and and $|\Gamma((\mathbf{p}'_n, \mathbf{q}'_n))-(\mathbf{p}'_n, \mathbf{q}'_n)|\longrightarrow 0$ are eventually close in the sense that $|(\mathbf{p}_n, \mathbf{q}_n)-(\mathbf{p}'_n, \mathbf{q}'_n)|\longrightarrow 0$.
\end{quote}
Thus, $\Gamma$ has sequentially at most one fixed point.

Summing up $p_i'$ from 1 to $m$, for each $i$
\[\sum_{i=1}^{m}p_i'(\mathbf{p}, \mathbf{q})=\frac{\sum_{i=1}^{m}p_i+\sum_{i=1}^{m}\max(M(a_i, \mathbf{q})-M(\mathbf{p}, \mathbf{q}),0)}{1+\sum_{k=1}^m\max(M(a_k, \mathbf{q})-M(\mathbf{p}, \mathbf{q}),0)}=1.\]
Similarly, summing up $q_j'$ from 1 to $n$, for each $j$
\[\sum_{j=1}^{n}q_j'(\mathbf{p}, \mathbf{q})=\frac{\sum_{j=1}^{n}q_j+\sum_{j=1}^{n}\max(M(\mathbf{p}, \mathbf{q})-M(\mathbf{p},b_j),0)}{1+\sum_{k=1}^n\max(M(\mathbf{p}, \mathbf{q})-M(\mathbf{p},b_k),0)}=1.\]
Let $\mathbf{p}'(\mathbf{p}, \mathbf{q})=(p_1', p_2', \dots, p_m')$, $\mathbf{q}'(\mathbf{p}, \mathbf{q})=(q_1', q_2', \dots, q_n')$. Then, $\Gamma=(\mathbf{p}'(\mathbf{p}, \mathbf{q}), \mathbf{q}'(\mathbf{p}, \mathbf{q}))$ is a uniformly continuous function from $P\times Q$ into itself. There are $m+n-2$ independent vectors in $P\times Q$, and so $P\times Q$ is an $m+n-2$-dimensional space. Since it is a product of two simplices, it is a compact subset of a metric space. Therefore, $\Gamma$ has a fixed point. Let $(\mathbf{\tilde{p}}, \mathbf{\tilde{q}})$ be the fixed point, and $\lambda=\sum_{k=1}^n\max(M(a_k, \mathbf{\tilde{q}})-M(\mathbf{\tilde{p}}, \mathbf{\tilde{q}}),0)$, $\lambda'=\sum_{k=1}^m\max(M(\mathbf{\tilde{p}}, \mathbf{\tilde{q}})-M(\mathbf{\tilde{p}},b_k),0)$. Then, 
\[\frac{\tilde{p}_i+\max(M(a_i, \mathbf{\tilde{q}})-M(\mathbf{\tilde{p}}, \mathbf{\tilde{q}}),0)}{1+\lambda}=\tilde{p}_i,\]
\[\frac{\tilde{q}_j+\max(M(\mathbf{\tilde{p}}, \mathbf{\tilde{q}})-M(\mathbf{\tilde{p}},b_j),0)}{1+\lambda'}=\tilde{q}_j.\]
Thus, we have
\[\max(M(a_i, \mathbf{\tilde{q}})-M(\mathbf{\tilde{p}}, \mathbf{\tilde{q}}),0)=\lambda \tilde{p}_i,\]
and
\[\max(M(\mathbf{\tilde{p}}, \mathbf{\tilde{q}})-M(\mathbf{\tilde{p}},b_j),0)=\lambda'\tilde{q}_j.\]
Since $M(\mathbf{\tilde{p}}, \mathbf{\tilde{q}})=\sum_{i=1}^m M(a_i,\mathbf{\tilde{q}})$, it is impossible that $\max(M(a_i, \mathbf{\tilde{q}})-M(\mathbf{\tilde{p}}, \mathbf{\tilde{q}}),0)=M(a_i, \mathbf{\tilde{q}})-M(\mathbf{\tilde{p}}, \mathbf{\tilde{q}})>0$ for all $i$ such that $\tilde{p}_i>0$. Therefore, $\lambda=0$, and we have
$\sup_{\mathbf{p}}M(\mathbf{p}, \mathbf{\tilde{q}})=M(\mathbf{\tilde{p}}, \mathbf{\tilde{q}})$. Similarly, we obtain $\lambda'=0$ and $\inf_{\mathbf{q}}M(\mathbf{\tilde{p}}, \mathbf{q})=M(\mathbf{\tilde{p}}, \mathbf{\tilde{q}})$. Then,
\[v_B^*=\inf_{\mathbf{q}}\sup_{\mathbf{p}}M(\mathbf{p}, \mathbf{\tilde{q}})\leq \sup_{\mathbf{p}}\inf_{\mathbf{q}}M(\mathbf{\tilde{p}}, \mathbf{q})=v_A^*.\]
With (\ref{e1}) we obtain
\[v_A^*=v_B^*.\]
Therefore, the value of the game is determined at the fixed point of $\Gamma$.
\begin{table}[htpb]
\begin{center}
\begin{tabular}{c|c|c|c|}
\multicolumn{4}{c}{\ \ \ \ \ \ \ \ \ \ \ \ \ \ \ \ Player 2} \\ \cline{2-4}
& &X&Y\\ \cline{2-4}
Player&X&1, -1 &-1, 1\\ \cline{2-4}
1&Y& -1, 1& 1, -1\\ \cline{2-4}
\end{tabular}
\vspace{.3cm}
\caption{Example of game}
\label{ga-1}
\end{center}
\end{table}

Consider an example. See a game in Table \ref{ga-1}. It is an example of the so-called Matching-Pennies game. Pure strategies of Player 1 and 2 are $X$ and $Y$. The left side number in each cell represents the payoff of Player 1 and the right side number represents the payoff of Player 2. Let $p_X$ and $1-p_X$ denote the probabilities that Player 1 chooses, respectively, $X$ and $Y$, and $q_X$ and $1-q_X$ denote the probabilities for Player 2. Denote the expected payoff of Player 1 by $M(p_X, q_X)$. Since we consider a zero-sum game, the expected payoff of Player 2 is $-M(p_X, q_X)$. Then,
\begin{align*}
M(p_X, q_X)=p_Xq_X-(1-p_X)q_X-p_X(1-q_X)+(1-p_X)(1-q_X)=(2p_X-1)(2q_X-1)
\end{align*}
Denote the payoff of Player 1 when he chooses $X$ by $M(X, q_X)$, and that when he chooses $Y$ by $M(Y, q_X)$. Similarly for Player B. Then,
\[M(X, q_X)=2q_X-1,\ M(Y, q_X)=1-2q_X,\ -M(p_X,X)=1-2p_X,\ -M(p_X,Y)=2p_X-1.\]
And we have
\[\mathrm{When\ }q_X>\frac{1}{2},\ M(X, q_X)>M(Y, q_X)\ \mathrm{and}\ M(X, q_X)>M(p_X,q_X)\ \mathrm{for}\ p_X<1,\]
\[\mathrm{When\ }q_X<\frac{1}{2},\ M(Y, q_X)>M(X, q_X)\ \mathrm{and}\ M(Y, q_X)>M(p_X,q_X)\ \mathrm{for}\ p_X>0,\]
\[\mathrm{When\ }p_X>\frac{1}{2},\ -M(p_X, Y)>-M(p_X, X)\ \mathrm{and}\ -M(p_X, Y)>-M(p_X,q_X)\ \mathrm{for}\ q_X>0,\]
\[\mathrm{When\ }p_X<\frac{1}{2},\ -M(p_X, X)>-M(p_X, Y)\ \mathrm{and}\ -M(p_X, X)>-M(p_X,q_X)\ \mathrm{for}\ q_X<1.\]

Consider sequences $(p_X(m))_{m\geq 1}$ and $(q_X(m))_{m\geq 1}$, and let $0<\varepsilon<\frac{1}{2}$, $0<\delta<\varepsilon$. There are the following cases.
\begin{enumerate}
\item 
\begin{enumerate}
	\item If $p_X(m)>\frac{1}{2}+\delta$ and $q_X(m)>\frac{1}{2}+\delta$, or
	\item $p_X(m)>\frac{1}{2}+\delta$ and $q_X(m)<\frac{1}{2}-\delta$, or
	\item $p_X(m)<\frac{1}{2}-\delta$ and $q_X(m)<\frac{1}{2}-\delta$, or
	\item $p_X(m)<\frac{1}{2}-\delta$ and $q_X(m)>\frac{1}{2}+\delta$, or
	\item $p_X(m)>\frac{1}{2}+\delta$ and $\frac{1}{2}-\varepsilon<q_X(m)<\frac{1}{2}+\varepsilon$, or
	\item $p_X(m)<\frac{1}{2}-\delta$ and $\frac{1}{2}-\varepsilon<q_X(m)<\frac{1}{2}+\varepsilon$, or
	\item $\frac{1}{2}-\varepsilon<p_X(m)<\frac{1}{2}+\varepsilon$, and $q_X(m)>\frac{1}{2}+\delta$ or
	\item $\frac{1}{2}-\varepsilon<p_X(m)<\frac{1}{2}+\varepsilon$, and $q_X(m)<\frac{1}{2}-\delta$ or
\end{enumerate}
then, there exists no pair of $(p_X(m),q_X(m))$ such that $M(X, q_X(m))-M(p_X(m), q_X(m))\longrightarrow 0$ and $-[M(p_X(m), Y)-M(p_X(m), q_X(m))]\longrightarrow 0$.
\item If $\frac{1}{2}-\varepsilon<p_X(m)<\frac{1}{2}+\varepsilon$ and $\frac{1}{2}-\varepsilon<q_X(m)<\frac{1}{2}+\varepsilon$ with $0<\varepsilon<\frac{1}{2}$, $M(X, q_X(m))-M(p_X(m), q_X(m))\longrightarrow 0$, $M(Y, q_X(m))-M(p_X(m), q_X(m))\longrightarrow 0$, $-[M(p_X(m), X)-M(p_X(m), q_X(m))]\longrightarrow 0$ and $-[M(p_X(m), Y)-M(p_X(m), q_X(m))]\longrightarrow 0$, then $(p_X(m), q_X(m))\longrightarrow (\frac{1}{2}, \frac{1}{2})$.
\end{enumerate}
Therefore, the payoff functions satisfy Assumption \ref{assump1}.

\bibliography{yasuhito}

\begin{thebibliography}{1}

\bibitem{berg}
J.~Berger, D.~Bridges, and P.~Schuster.
\newblock The fan theorem and unique existence of maxima.
\newblock {\em Journal of Symbolic Logic}, 71:713--720, 2006.

\bibitem{berger}
Josef Berger and Hajime Ishihara.
\newblock Brouwer's fan theorem and unique existence in constructive analysis.
\newblock {\em Mathematical Logic Quaterly}, 51(4):360--364, 2005.

\bibitem{br}
D.~Bridges and F.~Richman.
\newblock {\em Varieties of Constructive Mathematics}.
\newblock Cambridge University Press, 1987.

\bibitem{bv}
D.~Bridges and L.~V\^{i}\c{t}\u{a}.
\newblock {\em Techniques of Constructive Mathematics}.
\newblock Springer, 2006.

\bibitem{hirst}
Jeffry~L. Hirst.
\newblock Notes on reverse mathematics and {B}rouwer's fixed point theorem.
\newblock {\em http://www.mathsci.appstate.edu/\textasciitilde
  jlh/snp/pdfslides/bfp.pdf}, pages 1--6, 2000.

\bibitem{kel}
R.~B. Kellogg, T.~Y. Li, and J.~Yorke.
\newblock A constructive proof of {B}rouwer fixed-point theorem and
  computational results.
\newblock {\em SIAM Journal on Numerical Analysis}, 13:473--483, 1976.

\bibitem{orevkov}
V.~P. Orevkov.
\newblock A constructive mapping of a square onto itself displacing every
  constructive point.
\newblock {\em Soviet Math.}, 4:1253--1256, 1963.

\bibitem{da}
D.~van Dalen.
\newblock {B}rouwer's $\varepsilon$-fixed point from {S}perner's lemma.
\newblock {\em Theoretical Computer Science}, 412(28):3140--3144, June 2011.

\bibitem{veld}
W.~Veldman.
\newblock {B}rouwer's approximate fixed point theorem is equivalent to
  {B}rouwer's fan theorem.
\newblock In S.~Lindstr{\"{o}}m, E.~Palmgren, K.~Segerberg, and
  V.~Stoltenberg-Hansen, editors, {\em Logicism, Intuitionism and Formalism}.
  Springer, 2009.

\end{thebibliography}
\end{document}